\documentclass[amstex,a4]{article}
 \usepackage{amsfonts, amsmath, amsthm, bm, amssymb, color}

\newtheorem{thm}{Theorem}[section]
\newtheorem{cor}{Corollary}[section]

\newcommand{\ppp}{\partial}
\newcommand{\ep}{\varepsilon}
\newcommand{\www}{\widetilde}

\title{\textbf{
{Uniqueness in inverse boundary value problems
for fractional diffusion equations} }}

\author{\textbf{Zhiyuan Li$^1$},  \textbf{O.~Yu.~Imanuvilov}$^2$, 
\textbf{Masahiro Yamamoto$^3$}}

\date{}

\begin{document}
 \maketitle

\begin{center}
$^{1,3}$ Graduate School of Mathematical Sciences, the University of Tokyo, 
3-8-1 Komaba, Meguro-ku, Tokyo, 153-8914, Japan.\\
E-mail: zyli@ms.u-tokyo.ac.jp, myama@ms.u-tokyo.ac.jp\\
$^2$Department of Mathematics, Colorado State University, 
101 Weber Building, Fort Collins, CO, 80523-1874, U.S.A.\\
E-mail: oleg@math.colostate.edu.
\end{center}

\begin{abstract}
We consider an inverse boundary value problem for diffusion equations with 
multiple fractional time derivatives.
We prove the uniqueness in determining a number of fractional time-derivative 
terms, the orders of the derivatives and spatially varying coefficients.
\end{abstract}

\renewcommand{\thefootnote}{\fnsymbol{footnote}}

\footnotetext{
$^{1,3}$ supported by the Leading Graduate Course for 
Frontiers of Mathematical Sciences and Physics
(The University of Tokyo).\\
$^2$ Partially supported by NSF grant DMS 1312900
}

\textbf{Keywords:} fractional diffusion equation,
  inverse problem, determination of fractional orders,
  Dirichlet-to-Neumann map,  multi-time-fractional derivatives.


\section{Introduction}
Let $\Omega$ be an open bounded domain in $\mathbb{R}^d$ with smooth boundary, 
for example, of $C^2$-class,
$\nu$ be the outward unit normal vector to $\partial\Omega$.
We denote $\frac{\partial u}{\partial\nu}=\nabla u\cdot\nu$. 
Let $T>0$ be fixed arbitrarily.
Consider the following initial-boundary value problem
\begin{equation}\label{equ-u}
\left\{
  \begin{array}{*{20}c}
 \sum_{j=1}^{\ell} p_j(x)\partial^{\alpha_j}_t u(x,t)
 =\Delta u(x,t) + p(x)u(x,t), \quad (x,t)\in\Omega\times(0,T), \hfill\\
  {u(x,0)=0,\quad x\in \Omega,} \hfill \\
  {u\vert_{\partial\Omega} = \lambda(t)g(x), \quad 0 <  t < T,} \hfill
 \end{array} \right.
 \end{equation}
where $\alpha_j$, $j=1,\cdots,\ell$, are positive constants such that
\begin{align}\label{alpha}
0<\alpha_{\ell}<\cdots<\alpha_1<1.
\end{align}
Here and henceforth, for $\alpha\in(0,1)$,
by $\partial_t^{\alpha}v$ we denote the Caputo fractional
derivative with respect to $t$:
\begin{equation}\label{def-Caputo}
\partial_t^{\alpha} v(t) = \frac{1}{\Gamma(1-\alpha)}
\int_0^t (t-\tau)^{-\alpha}\frac{d}{d\tau} v(\tau) d\tau
\end{equation}
(e.g., Podlubny \cite{Podlubny}) and $\Gamma$ is the Gamma function.

The case of $\ell=1$, i.e., a single-term time-fractional diffusion
equation is used for example as a model equation for 
the anomalous diffusion phenomena in heterogeneous media
(e.g., Metzler and Klafter \cite{MK}).  We further refer to 
Kilbas, Srivastava and Trujillo \cite{KST}, Luchko \cite{Lu2009JMAA},
Luchko and Gorenflo \cite{LG99}, Mainardi \cite{Mai}, Podlubny \cite{Podlubny},
Sakamoto and Yamamoto \cite{SY}.

On the other hand, diffusion equations whose orders of the derivatives change
in time and/or spatial coordinates, are proposed as feasible models
(e.g., Chechkin, Gorenflo and Sokolov \cite{CGS}, Coimbra \cite{Co},
Lorenzo and Hartley \cite{LM}, Mainardi, Mura, Pagnini and Gorenflo 
\cite{MMPG}, Pedro, Kobayshi, Pereira and Coimbra \cite{PKPC},
Sokolov, Chechkin and Klafter \cite{SCK}.)
Among them, we consider a multi-term time-fractional diffusion equation 
(\ref{equ-u}).

For applying (\ref{equ-u}) as model equation, in order to interpret 
measurement data, we usually need to suitably choose
$\ell, p_j, \alpha_j, p$ which describe physical properties of the 
diffusion process under consideration.
This is our inverse problem, and we discuss the uniqueness as
the fundamental theoretical topic for the inverse problem.

Henceforth, for $\ell \in \Bbb N$, we set $\vec{\alpha} 
= (\alpha_1, ..., \alpha_{\ell}) \in (0,1)^{\ell}$ where
$\alpha_{\ell} < \alpha_{\ell-1} < \cdots < \alpha_1$.
We note that also $\ell$ is unknown parameter in the inverse problem.

We state\\
{\bf Inverse problem}
Let $\lambda \not\equiv 0$ be fixed.
For $g\in H^{\frac{3}{2}}(\partial\Omega)$, 
we define the Dirichlet-to-Neumann 
map by
$$
\Lambda(\ell,\vec{\alpha},p_j,p)g:=\frac{\partial u}{\partial \nu}|
_{\partial\Omega\times(0,T)} \in L^2(0,T;H^{\frac{1}{2}}(\ppp\Omega)).
$$
Can we uniquely determine $(\ell,\vec{\alpha},p_j,p)$
by the map $\Lambda(\ell,\vec{\alpha},p_j,p): H^{\frac{3}{2}}(\ppp\Omega)
\longrightarrow L^2(0,T; H^{\frac{1}{2}}(\ppp\Omega))$?

In Section 2, we prove that the Dirichlet-to-Neumann map is well
defined.  Our inverse problem is based on the Dirichlet-to-Neumann map,
and for elliptic equations, there have been numerous important works.
Here we do not intend any lists of references and we refer only to 
Imanuvilov and Yamamoto \cite{IYAM}, Isakov \cite{Is},
Sylvester and Uhlmann \cite{SU} and the references therein.

For the statement of our main results, we introduce some notations.
As an admissible set of unknown fractional orders including numbers 
and coefficients,
we set
$$
\mathcal{U} = \{(\ell, \vec{\alpha}, p_1, ..., p_{\ell}, p)
\in \Bbb N \times (0,1)^{\ell}\times C^{\infty}(\overline{\Omega})
^{\ell+1}; \thinspace p_j \ge 0, \not\equiv 0, \thinspace
j=2,3,..., \ell, \thinspace p_1 > 0, \thinspace p\le 0 
\thinspace \mbox{on} \thinspace \overline\Omega \}.
$$
where $\vec{\alpha}:=(\alpha_1,\cdots,\alpha_\ell)$ such that $\alpha_\ell<\alpha_{\ell-1}<\cdots<\alpha_1$.
For $\theta \in \left(0, \frac{\pi}{2}\right)$, we further set
$$
\Omega_{\theta} :=\{z\in\mathbb{C};\thinspace z\neq 0, \thinspace
|\arg z|<\theta \}.
$$

We are ready to state our main result.
\begin{thm}[Uniqueness]\label{thm-uniqu}
Let $(\ell,\vec{\alpha},p_j,p) \in \mathcal{U}$ and $(m,\vec{\beta},q_j,q) \in
\mathcal{U}$.
Assume that for some $\theta\in (0,\frac\pi 2)$ the function 
$\lambda \not\equiv 0$ 
can be analytically extended to $\Omega_{\theta}$ with $\lambda(0)=0$
and $\lambda'(0)=0$ and there exists a constant $C_0>0$ such that 
$\vert \lambda^{(k)}(t) \vert \le C_0e^{C_0t}$, $t>0$, $0\le k \le 2$.
Then $\ell=m$, $\vec{\alpha}=\vec{\beta}$, $p_j=q_j$, $1\le j\le \ell$ and $p=q$ provided
\begin{align}\label{DN}
\Lambda(\ell,\vec{\alpha},p_j,p)g = \Lambda(m,\vec{\beta},q_j,q)g, \quad
g \in H^{\frac{3}{2}}(\partial\Omega).
\end{align}
\end{thm}

The assumption $p\le 0$ on $\overline{\Omega}$ is necessary for proving
that $\vert u(x,t)\vert = O(e^{C_1t})$ as $t \to \infty$ with some
constant $C_1>0$. Such an estimate is sufficient for taking the Laplace
transforms of $u$, which is a key of the proof of Theorem 1.1.
In this paper, we do not discuss the inverse problem without the condition
$p\le 0$.
 
In $\mathcal{U}$, we can relax the regularity of $p, p_1, ...., 
p_{\ell}$ but we do not discuss here.
Moreover, in the two dimensional case of $d=2$, thanks to Imanuvilov and 
Yamamoto \cite{IY}, we can prove a sharp uniqueness result
where Dirichlet inputs and Neumann outputs can be restricted an arbitrary
subboundary and the required regularity of unknown coefficients is relaxed.

\begin{cor}\label{buharik}
Let $\Omega\subset \Bbb R^2$ be a bounded domain with smooth boundary 
$\ppp\Omega$ and $\Gamma \subset \ppp\Omega$ be an arbitrarily given 
subboundary and let $\gamma >2$ be arbitrarily fixed.
We assume the $\lambda$ satisfies the same conditions as in Theorem \ref{thm-uniqu}. 
We set
$$
\mathcal{\widehat{U}} = \{(\ell, \vec{\alpha}, p_1, ..., p_{\ell}, p)
\in \Bbb N \times (0,1)^{\ell} \times W^{2,\infty}(\Omega) \times
(W^{1,\gamma}(\Omega))^{\ell};
\thinspace p_j \ge 0, p_j\not\equiv 0, \thinspace
j=2,3,..., \ell, \thinspace p_1 > 0, \thinspace p\le 0 
\thinspace \mbox{on $\overline\Omega$}\}.
$$
If $(\ell,\vec{\alpha},p_j,p), (m,\vec{\beta},q_j,q) \in \mathcal{\widehat U}$ 
satisfy 
$$
\Lambda(\ell,\vec{\alpha},p_j,p)g = \Lambda(m,\vec{\beta},q_j,q)g \quad
\mbox{on $\Gamma$}
$$
for all $g \in H^{\frac{3}{2}}(\ppp\Omega)$ with 
supp $g \subset \Gamma$, then $\ell=m$, $\vec{\alpha}=\vec{\beta}$, $p_j=q_j$, 
$1\le j\le \ell$ and $p=q$.
\end{cor}

As for single-term time-fractional diffusion equations, 
there are not many works on inverse problems in spite of the physical 
and practical importance and see e.g., Cheng, Nakagawa, Yamamoto 
and Yamazaki \cite{CNYY}, Li, Zhang, Jia and Yamamoto \cite{LZJY},
Hatano, Nakagawa, Wang and Yamamoto \cite{HNWY}.
Moreover for inverse problems for 
multi-term time-fractional diffusion equations, to the best knowledge of 
the authors, there are no existing results.

The rest of the paper is organized as follows.  In Section 2, we prove
properties of solutions to (\ref{equ-u}) which are necessary for the proof of
Theorem \ref{thm-uniqu}.  In particular, the $t$-analyticity of solution is essential.
In Section 3, by applying the Laplace transforms of the solutions to (\ref{equ-u})
and reducing our inverse problem to the inverse boundary value problem for
elliptic equations, we complete the proof of Theorem \ref{thm-uniqu}.

\section{Forward problem}
For $\theta \in \left(0, \frac{\pi}{2}\right)$ and $T>0$, we set
$$
\Omega_{\theta} :=\{z\in\mathbb{C};\thinspace z\neq 0, \thinspace
|\arg z|<\theta\},\quad 
\Omega_{\theta,T} :=\{z\in \Omega_{\theta};\thinspace
\vert z\vert < T\}.
$$

In this section, we establish the analyticity of the solution $u$ to
the initial- boundary value problem (\ref{equ-u}) as well as the unique
existence of the solution.  As for other results for solutions to (\ref{equ-u}), see 
Beckers and Yamamoto \cite{BY}, Li and Yamamoto \cite{LY},
Li, Liu and Yamamoto \cite{LLY}, Luchko \cite{Lu2009}, \cite{Luchko2}
for example.

\begin{thm}\label{thm-analy}
Let $(\ell,\vec{\alpha},p_j,p) \in \mathcal{U}$ and $T>0$ be arbitrarily given.
Assume that
$g \in H^{\frac{3}{2}}(\ppp\Omega)$, $\lambda(0)=0$, 
$\lambda'(0) = 0$, 
for $\theta\in (0,\frac\pi 2)$ and $T>0$, the function $\lambda(t)$ can be 
analytically extended to $\Omega_{\theta}$ and
$\lambda \in W^{2,\infty}(\Omega_{\theta,T})$.
Then there exists a unique mild solution $u\in C([0,\infty);H^2(\Omega))$ and
$Au(t)$: $(0,\infty)\rightarrow H^2(\Omega)$ can be analytically extended to
$\Omega_{\theta}$.

Moreover, for $g \in C^{\infty}(\ppp\Omega)$ and any $T>0$, we have  
\begin{equation}\label{esti-u_C}
\Vert u\Vert_{C(\overline{\Omega}\times [0,T])} 
\le \Vert g\Vert_{C(\ppp\Omega)}\Vert \lambda\Vert_{C[0,T]}.
\end{equation}
\end{thm}
\begin{proof}
The proof is based on the following observation. By  the Sobolev extension 
theorem,  $g\in H^{\frac{3}{2}}(\partial\Omega)$ allows us to choose 
$\widetilde{g}\in H^2(\Omega)$ such that $\widetilde{g}|_{\partial\Omega}=g$.
Introducing the new unknown function $\widetilde{u}(x,t)
= u(x,t)-\lambda(t)\widetilde{g}(x)$, we can rewrite 
(\ref{equ-u}) as
\begin{equation}\label{equ-u'}
\left\{\begin{array}{*{20}c}
\partial_t^{\alpha_1}\widetilde{u}+\sum_{j=2}^{\ell} 
\widetilde{p}_j(x)\partial^{\alpha_j}_t \widetilde{u}
= {\rm div}(\frac{1}{p_1(x)}\nabla\widetilde{u})+B(x)\cdot\nabla\widetilde{u} 
+ b(x)\widetilde{u} +F(x,t),\mbox{ in } \Omega\times(0,T),\hfill\\
  {\widetilde{u}(x,0)=0,\quad x\in \Omega,} \hfill \\
  {\widetilde{u}\vert_{\partial\Omega} = 0, \quad 0 <  t < T,} \hfill
 \end{array} \right.
 \end{equation}
where $\www{p}_j(x):=\frac{p_j(x)}{p_1(x)}$, $j=2,\cdots,\ell$, 
$B(x):=\nabla(\frac{-1}{p_1(x)})$, 
$b(x):=\frac{p(x)}{p_1(x)}$ and
\begin{equation}
  F(x,t)
:=\frac{1}{p_1(x)}\left(\lambda(t)\Delta\www{g}(x) 
+ \lambda(t)p(x)\www{g}(x) 
- \sum_{j=1}^{\ell}(\ppp_t^{\alpha_j}\lambda)(t)p_j(x)\www{g}(x)\right).
\end{equation}
Then $F(x,\cdot)$ can be analytically extended to
$\Omega_{\theta}$.  In fact, it is sufficient to prove that
$\ppp^{\alpha}_t\lambda$ can be analytically extended to 
$\Omega_{\theta, T}$ with any $\alpha \in (0,1)$ and $T>0$.
Let $z \in \Omega_{\theta,T}$ be arbitrarily fixed.
We set 
$$
\lambda_{\alpha}(z) := \frac{z^{-\alpha-1}}{\Gamma(1-\alpha)}
\int^1_0 (1-\eta)^{-\alpha}\frac{\ppp(\lambda(\eta z))}{\ppp\eta}d\eta
= \frac{1}{\Gamma(1-\alpha)}
\int^z_0 (z-s)^{-\alpha}\frac{d\lambda(s)}{ds}ds.
$$
Here the integral is considered on the segment from $0$ to
$z$ in $\mathbb{C}$.  Then we can see that $\lambda_{\alpha}(t)
= \ppp_t^{\alpha}\lambda(t)$ for $t>0$.  For any small $\ep>0$, we set
$$
\lambda_{\alpha}^{\ep}(z) =
 \frac{z^{-\alpha-1}}{\Gamma(1-\alpha)}
\int^{1-\ep}_0 (1-\eta)^{-\alpha}\frac{\ppp(\lambda(\eta z))}{\ppp\eta}
d\eta.
$$
By the analyticity of $\lambda$ in $\Omega_{\theta}$, we see that
$\lambda_{\alpha}^{\ep}$ is analytic in $\Omega_{\theta,T}$.
Let $\ep_0>0$ be an arbitrarily fixed small constant.
For any $z \in \Omega_{\theta,\ep_0,T}
:=\{z\in\Omega_{\theta,T}; |z|>\ep_0\}$, we have
$$
| \lambda_{\alpha}^{\ep}(z) - \lambda_{\alpha}(z)|
\le \frac{\ep_0^{-\alpha-1}}{\Gamma(1-\alpha)}
\int^1_{1-\ep} (1-\eta)^{-\alpha}\left|
\frac{\ppp(\lambda(\eta z))}{\ppp\eta}\right| d\eta
\le \frac{\ep_0^{-\alpha-1}}{\Gamma(1-\alpha)}
\sup_{0<\eta<1}
\left| \frac{\ppp(\lambda(\eta z))}{\ppp\eta}\right|
\int^1_{1-\ep} (1-\eta)^{-\alpha} d\eta.
$$
Therefore, for any fixed $\ep_0>0$ and $T>0$, we see that
$$
\sup_{z \in \Omega_{\theta,\ep_0,T}} | \lambda_{\alpha}
^{\ep}(z) - \lambda_{\alpha}(z)| \longrightarrow 0
\quad\mbox{as}\,\, \ep \to 0.
$$  
Since $\lambda_{\alpha}^{\ep}$ is analytic in
$\Omega_{\theta,T}$, we see that $\lambda_{\alpha}$ is analytic in
$\Omega_{\theta,T}$, because $\lambda_{\alpha}$ is the uniform convergent
limit of analytic $\lambda_{\alpha}^{\ep}$ in any compact subset of
$\Omega_{\theta,T}$.
Thus we completed the proof that $F(\cdot,t)$ can be analytically extended
to $\Omega_{\theta}$.

Next we estimate $F$.  Let $T\ge 1$.  First we have
$$
\Vert F\Vert_{L^{\infty}(0,T;L^2(\Omega))}  
\le C\left(\Vert \lambda\Vert_{L^{\infty}(0,T)}
+ \sum_{j=1}^{\ell} \Vert \ppp_t^{\alpha_j}\lambda\Vert_{L^{\infty}(0,T)}
\right).
$$
Here and henceforth $C>0$ denotes a generic constant which is
independent of $T, t>0$, $z\in \Omega_{\theta}$, but dependent on 
$d, \Omega, g, \theta, p, p_1, ..., p_{\ell}, \alpha_1, ..., \alpha_{\ell}$.  
We have
$$
\vert \ppp_t^{\alpha_j}\lambda(t)\vert 
= \frac{1}{\Gamma(1-\alpha_j)}\left\vert 
\int^t_0 (t-s)^{-\alpha_j}\frac{d\lambda}{ds}(s) ds\right\vert
\le C\int^t_0 (t-s)^{-\alpha_j} ds \Vert \lambda\Vert_{C^1[0,T]}
\le CT\Vert \lambda\Vert_{C^1[0,T]},
$$
and so 
$$
\Vert F\Vert_{L^{\infty}(0,T;L^2(\Omega))}
\le CT\Vert \lambda\Vert_{C^1[0,T]}.
$$
Moreover, by $0 < \alpha_j < 1$, $\lambda'(0) = 0$ and 
integration by parts yield
\begin{eqnarray*}
&& \ppp_t^{\alpha_j}\lambda(t) 
= \frac{1}{\Gamma(1-\alpha_j)}\int^t_0 (t-s)^{-\alpha_j}
\frac{d\lambda}{ds}(s) ds\\
=&& \frac{1}{\Gamma(1-\alpha_j)}\left(
\left[ \lambda'(s)\frac{(t-s)^{1-\alpha_j}}{1-\alpha_j}
\right]^{s=0}_{s=t}
+ \int^t_0 \frac{(t-s)^{1-\alpha_j}}{1-\alpha_j}
\frac{d^2\lambda}{ds^2}(s) ds\right)\\
=&& \frac{1}{\Gamma(1-\alpha_j)}
\int^t_0 \frac{(t-s)^{1-\alpha_j}}{1-\alpha_j}
\frac{d^2\lambda}{ds^2}(s) ds.
\end{eqnarray*}
Therefore
$$
\ppp_t\ppp_t^{\alpha_j}\lambda(t) 
= \frac{1}{\Gamma(1-\alpha_j)}\int^t_0 (t-s)^{-\alpha_j}
\frac{d^2\lambda}{ds^2}(s) ds,
$$
and so
$$
\Vert \ppp_t\ppp_t^{\alpha_j}\lambda\Vert_{L^{\infty}(0,T)}
\le C\int^t_0 (t-s)^{-\alpha_j}ds \Vert \lambda\Vert_{C^2[0,T]}
\le CT\Vert \lambda\Vert_{C^2[0,T]}.
$$
Hence $\Vert \ppp_tF\Vert_{L^{\infty}(0,T;L^2(\Omega))}
\le CT\Vert \lambda\Vert_{C^2[0,T]}$.  Consequently
\begin{equation}
\Vert F\Vert_{W^{1,\infty}(0,T;L^2(\Omega))}
\le CT\Vert \lambda\Vert_{C^2[0,T]}
\end{equation}
for all $T\ge 1$.  Next we estimate $\Vert F(\cdot, z)\Vert_{L^2(\Omega)}$
for $z \in \Omega_{\theta,T}$.  Noting that 
$\www\lambda(\eta) = \lambda(\eta z)$ and 
$\frac{d\widetilde\lambda}{d\eta}(\eta) = z\frac{d\lambda}{d\eta}(\eta z)$
for $0 < \eta < 1$ and $z\in \Omega_{\theta,T}$, and 
$$
\lambda_{\alpha_j}(z) := \frac{1}{\Gamma(1-\alpha_j)}\int^z_0 (z-s)^{-\alpha_j}
\frac{d\lambda}{ds}(s) ds
= \frac{z^{-\alpha_j}}{\Gamma(1-\alpha_j)}\int^1_0 (1-\eta)^{-\alpha_j}
\frac{d\widetilde\lambda}{d\eta}(\eta) d\eta
$$
we have
\begin{eqnarray*}
&& \Vert \lambda_{\alpha_j}\Vert_{L^{\infty}(\Omega_{\theta,T})}
\le C\vert z\vert^{-\alpha_j}\int^1_0 (1-\eta)^{-\alpha_j} d\eta
\vert z\vert \sup_{s \in [0,z]} \left\vert \frac{d\lambda}{ds}(s)
\right\vert\\
\le && C\vert z\vert^{1-\alpha_j}\Vert \lambda\Vert_{W^{1,\infty}
(\Omega_{\theta,T})} 
\le CT\Vert \lambda\Vert_{W^{1,\infty}(\Omega_{\theta,T})}.
\end{eqnarray*}
Here $[0,z]$ denotes the closed segment in $\Bbb C$ from $0$ to $z$.
Arguing similarly to the case of $t \in [0,T]$, we obtain
\begin{equation}
\Vert F\Vert_{W^{1,\infty}(\Omega_{\theta,T};L^2(\Omega))}
\le CT\Vert \lambda\Vert_{W^{2,\infty}(\Omega_{\theta,T})}.
\end{equation}

Next we define operator $A$ in $H^2(\Omega)\cap H_0^1(\Omega)$ to be
\begin{equation*}
 (A\psi)(x)=-{\rm div}(\tfrac{1}{p_1(x)}\nabla\psi(x)),\quad x\in\Omega,\quad
 \psi\in H^2(\Omega)\cap H_0^1(\Omega).
\end{equation*}
Here and henceforth $\{\lambda_k,\phi_k\}_{k=1}^{\infty}$ denotes the 
eigensystem of the elliptic operator $A$ such that
$0<\lambda_1<\lambda_2\leq \lambda_3\cdots,\ \lim_{k\rightarrow\infty}
\lambda_k=\infty$, $A\phi_k=\lambda_k\phi_k$ 
and $\{\phi_k\}\subset H^2(\Omega)\cap H_0^1(\Omega)$ 
forms an orthonormal basis of $L^2(\Omega)$.
Then we can define the fractional power $A^{\gamma}$ for $\gamma > 0$ of
the operator $A$ (e.g., Tanabe \cite{Ta}), and we see that
$$
D(A^{\gamma})
= \left\{\psi\in L^2(\Omega): \sum_{n=1}^{\infty} \lambda_n^{2\gamma} |
(\psi,\phi_n)_{L^2(\Omega)}|^2<\infty\right\}
$$
is a Hilbert space with the norm
$$
\|\psi\|_{D(A^{\gamma})}=\left( \sum_{n=1}^{\infty} \lambda_n^{2\gamma} |
(\psi,\phi_n)_{L^2(\Omega)}|^2 \right)^{\frac{1}{2}}.
$$
We further define the operator $S(t):L^2(\Omega)\rightarrow L^2(\Omega)$ 
for $t>0$ by
\begin{equation}\label{def-S}
S(t)a:=\sum_{n=1}^{\infty} (a,\phi_n)_{L^2(\Omega)} 
E_{\alpha_1,1}(-\lambda_nt^{\alpha_1}) \phi_n \mbox{ in } L^2(\Omega)
\end{equation}
for $a\in L^2(\Omega)$, where $E_{\alpha,\beta}(z)$ is Mittag-Leffler 
function defined by
$$
E_{\alpha,\beta}(z):=\sum_{k=0}^{\infty}\frac{z^k}{\Gamma(\alpha k+\beta)},
\ z\in\mathbb{C}, \mbox{ $\alpha>0$, $\beta\in\mathbb{R}$.}
$$
The above formula and the classical asymptotics
\begin{equation}\label{esti-Gamma}
\Gamma(\eta) = e^{-\eta}\eta^{\eta-\frac{1}{2}}(2\pi)^{\frac{1}{2}}
\left(1 + O\left(\frac{1}{\eta}\right)\right) \quad \mbox{as}\quad \eta
\rightarrow +\infty, \thinspace \eta > 0
\end{equation}
(e.g., Abramowitz and Stegun \cite{AS}, p.257) imply that
the radius of convergence is $\infty$ and so
$E_{\alpha,\beta}(z)$ is an entire function of $z\in\mathbb{C}$.
 
Moreover the term-wise differentiations are possible and give
\begin{align*}
S'(t)a
&:=-\sum_{n=1}^{\infty}\lambda_n(a,\phi_n)_{L^2(\Omega)}
t^{\alpha_1-1}E_{\alpha_1,\alpha_1}(-\lambda_nt^{\alpha_1})\phi_n\
\text{in}\ L^2(\Omega)\\
S''(t)a
&:=-\sum_{n=1}^{\infty}\lambda_n(a,\phi_n)_{L^2(\Omega)}t^{\alpha_1-2}
E_{\alpha_1,\alpha_1-1}(-\lambda_nt^{\alpha_1})\phi_n\
\text{in}\ L^2(\Omega)
\end{align*}
for $a\in L^2(\Omega)$, $t>0$ (e.g., Podlubny \cite{Podlubny}).

From the definition of (\ref{def-S}) and the property of Mittag-Leffler 
function, $S'(z)$ and $S''(z)$ are analytic in the sector $\Omega_{\theta}$
and by Theorem 1.6 in \cite{Podlubny} (p.35), we can prove that there exists a 
constant $C>0$, which is independent of $z$ such that
\begin{align}
\Vert S(z)\Vert_{L^2(\Omega)\rightarrow L^2(\Omega)}&\leq C, 
\quad z\in \Omega_{\theta}, \label{esti-S}\\
\|A^{\gamma-1}S'(z)\|_{L^2(\Omega)\rightarrow L^2(\Omega)}&\leq 
C |z|^{\alpha_1-1-\alpha_1 \gamma},\quad z\in \Omega_{\theta},
\quad 0\leq\gamma\leq 1, \label{esti-S'}\\
\|A^{\gamma-1}S''(z)\|_{L^2(\Omega)\rightarrow L^2(\Omega)}
&\leq C |z|^{\alpha_1-2-\alpha_1 \gamma},\quad z\in \Omega_{\theta},
\quad 0\leq\gamma\leq 1 \label{esti-S''}.
\end{align}
We verify only (\ref{esti-S}), because the proofs of (13) and (14) are 
similar.
In fact, an estimate of $E_{\alpha,1}(-t)$ (e.g., Theorem 1.6 in 
\cite{Podlubny}, p.35) implies that (\ref{def-S}) holds for 
$t \in \Omega_{\theta}$ and
we have
\begin{eqnarray*}
&& \Vert S(z)a\Vert_{L^2(\Omega)}^2 
= \sum_{n=1}^{\infty} (a,\phi_n)_{L^2(\Omega)}^2
\vert E_{\alpha,1}(-\lambda_nz^{\alpha_1})\vert^2\\
\le&& C\sum_{n=1}^{\infty} (a,\phi_n)_{L^2(\Omega)}^2
\frac{1}{1+\vert \lambda_nz^{\alpha_1}\vert}
\le C \sum_{n=1}^{\infty} (a,\phi_n)_{L^2(\Omega)}^2 
= C\Vert a\Vert_{L^2(\Omega)}^2,
\end{eqnarray*}
which proves (\ref{esti-S}).

In view of $1>\alpha_1> ... > \alpha_{\ell} > 0$, regarding
$-\sum_{j=2}^{\ell} \widetilde{p}_j\ppp_t^{\alpha_j}\widetilde u
+ B\cdot \nabla \widetilde u + b\widetilde u$ as non-homogeneous term
in (\ref{equ-u'}), and we have
\begin{align*}
\www{u}(t) 
=&-\int^t_0 A^{-1}S'(t-s)(B\cdot \nabla \www{u}(s)
  + b\www{u}(s) + F(s)) ds\\
+& \sum^{\ell}_{j=2} \int^t_0 A^{-1}S'(t-s)\www{p}_j\ppp_t^{\alpha_j}
\www{u}(s) ds.
\end{align*}
Now we calculate the right-hand side.
Noting by the definition of Caputo fractional 
derivative \label{def-Caputo} and the Fubini theorem,
similarly to \cite{BY} or \cite{SY}, we change the orders of integrations 
to derive
\begin{align*}
&\int^t_0 A^{-1}S'(t-s) 
    \Big(\widetilde{p}_j\ppp_t^{\alpha_j}\widetilde u(s)\Big)ds
=\int_0^t A^{-1}S'(t-s) \frac{1}{\Gamma(1-\alpha_j)}
\left(\int_0^s (s-r)^{-\alpha_j} \www{p}_j\frac{d\www{u}}{dr}(r)dr\right) ds\\
=&\int_0^t \left(\int_r^t A^{-1}S'(t-s) 
     \frac{(s-r)^{-\alpha_j}}{\Gamma(1-\alpha_j)} ds \right) 
     \www{p}_j\frac{d\www{u}}{dr}(r)dr\\
=&\int_0^t \left(\int_0^{t-r} A^{-1}S'(t-r-\xi) 
     \frac{\xi^{-\alpha_j}}{\Gamma(1-\alpha_j)} d\xi \right) 
     \www{p}_j\frac{d\www{u}}{dr}(r)dr.
\end{align*}
Here in the last equality, we used the change of variable $\xi:=s-r$. 
Decomposing the integrand, we obtain
\begin{align*}
&\int_0^t \left(\int_0^{t-r} A^{-1}S'(t-r-\xi) 
     \xi^{-\alpha_j} d\xi \right) \www{p}_j \frac{d\www{u}}{dr}(r)dr\\
=&\int_0^t \left(\int_0^{t-r} A^{-1}S'(t-r-\xi)
     \left(\xi^{-\alpha_j} - (t-r)^{-\alpha_j}\right) d\xi \right) 
     \www{p}_j\frac{d\www{u}}{dr}(r)dr\\
&+\int_0^t \left(\int_0^{t-r} A^{-1}S'(t-r-\xi) d\xi \right) (t-r)^{-\alpha_j}
     \www{p}_j\frac{d\www{u}}{dr}(r)dr=:I_1(t)+I_2(t).
\end{align*}
Here we should understand that 
$$
I_1(t) = \lim_{\ep_1, \ep_2 \downarrow 0}\int^t_0
\left(\int^{t-r-\ep_1}_{\ep_2} 
A^{-1}S'(t-r-\xi)\left(\xi^{-\alpha_j} - (t-r)^{-\alpha_j}\right) d\xi \right) 
\www{p}_j\frac{d\www{u}}{dr}(r)dr,
$$
but throughout the following calculations, we can prove that the resulting 
integrals are all convergent, so that we present the calculations
without such passage to limits.

Integration by parts yields
\begin{align*}
I_1(t)
=& \int_0^{t-r} A^{-1}S'(t-r-\xi)
\left(\xi^{-\alpha_j} - (t-r)^{-\alpha_j}\right)\www{p}_j \www{u}(r)d\xi
\Big|_{r=0}^{r=t}\\
&+ \int_0^t \left(\int_0^{t-r} A^{-1}S''(t-r-\xi)
\left(\xi^{-\alpha_j} - (t-r)^{-\alpha_j}\right) d\xi \right) 
\www{p}_j\www{u}(r)dr\\
&+ \alpha_j\int_0^t \left(\int_0^{t-r} A^{-1}S'(t-r-\xi)d\xi\right)
    (t-r)^{-\alpha_j-1} \www{p}_j \www{u}(r)dr\\
&+ \int_0^t \lim_{\xi\to t-r} A^{-1}S'(t-r-\xi)
     \left(\xi^{-\alpha_j} - (t-r)^{-\alpha_j}\right)\www{p}_j \www{u}(r)dr.
\end{align*}
From (\ref{esti-S'}) and the fact $\alpha_1>\alpha_j$ for
$j=2,\cdots,\ell$, we deduce
\begin{align*}
&\left\|\int_0^{t-r} A^{-1}S'(t-r-\xi)
\left(\xi^{-\alpha_j} - (t-r)^{-\alpha_j}\right)d\xi\right\|
_{L^2(\Omega)\to L^2(\Omega)}\\
\le &\left\|\int_0^{t-r} A^{-1}S'(t-r-\xi)
\xi^{-\alpha_j}d\xi\right\|_{L^2(\Omega)\to L^2(\Omega)}
+ \left\|\int_0^{t-r} A^{-1}S'(t-r-\xi)
(t-r)^{-\alpha_j}d\xi\right\|_{L^2(\Omega)\to L^2(\Omega)}\\
\le& C\int^{t-r}_0 (t-r-\xi)^{\alpha_1-1}\xi^{(1-\alpha_j)-1} d\xi
+ C\int^{t-r}_0 (t-r-\xi)^{\alpha_1-1}d\xi (t-r)^{-\alpha_j}\\
\le& C(t-r)^{\alpha_1-\alpha_j}\frac{\Gamma(\alpha_1)\Gamma(1-\alpha_j)}
{\Gamma(\alpha_1+1-\alpha_j)}
+ \frac{C}{\alpha_1}(t-r)^{\alpha_1-\alpha_j} \longrightarrow 0\quad\mbox{ as}\,\,r \to t.
\end{align*}

Moreover, by (\ref{esti-S'}) and $\vert \xi^{-\alpha_j} - (t-r)^{-\alpha_j}
\vert 
\le C\xi^{-\alpha_j-1}\vert t-r-\xi\vert$ for $0 < \xi < t-r$, we have
$$
\|A^{-1}S'(t-r-\xi)(\xi^{-\alpha_j}-(t-r)^{-\alpha_j})\|_{L^2(\Omega)\to 
L^2(\Omega)}
\leq C (t-r-\xi)^{\alpha_1} \xi^{-\alpha_j-1}\to 0\thinspace 
\mbox{as $\xi\to t-r$}
$$
for $t>r$.  Therefore, by $\www{u}(0)=0$, we obtain
\begin{align*}
I_1(t)
=&\int_0^t \left(\int_0^{t-r} A^{-1}S''(t-r-\xi)
\left(\xi^{-\alpha_j} - (t-r)^{-\alpha_j}\right) d\xi \right) 
\www{p}_j\www{u}(r)dr\\
&+\alpha_j\int_0^t A^{-1}(S(t-r)-S(0))(t-r)^{-\alpha_j-1} \www{p}_j 
\www{u}(r)dr.
\end{align*}
Again by integration by parts, we find
\begin{align*}
I_2(t)
=&\int_0^t A^{-1}(S(t-r)-S(0))(t-r)^{-\alpha_j}\www{p}_j\frac{d\www{u}}{dr}(r)
dr\\
=&A^{-1}(S(t-r)-S(0)) (t-r)^{-\alpha_j}\www{p}_j\www{u}(r)\Big|_{r=0}^{r=t}
 +\int_0^t A^{-1}S'(t-r)(t-r)^{-\alpha_j}\www{p}_j\www{u}(r)dr\\
&-\alpha_j\int_0^t A^{-1}(S(t-r)-S(0))(t-r)^{-\alpha_j-1}\www{p}_j\www{u}(r)dr.
\end{align*}
Now (\ref{esti-S'}) yields
\begin{align*}
&\|A^{-1}(S(t-r)-S(0)) (t-r)^{-\alpha_j}\|_{L^2(\Omega)\to L^2(\Omega)}
=\left\|\int_0^{t-r} A^{-1}S'(\xi) d\xi \right\|_{L^2(\Omega)\to L^2(\Omega)}
(t-r)^{-\alpha_j}\\
\leq& C\int_0^{t-r} \xi^{\alpha_1-1}d\xi (t-r)^{-\alpha_j}
\leq C(t-r)^{\alpha_1-\alpha_j}\to0, \mbox{ as $r\to t$.}
\end{align*}
Consequently, using $\www{u}(0) = 0$, we find
\begin{align*}
 &\int^t_0 A^{-1}S'(t-s) 
    \Big(\widetilde{p}_j\ppp_t^{\alpha_j}\widetilde u(s)\Big)ds\\
=&\int^t_0 A^{-1}S'(t-r)(t-r)^{-\alpha_j} 
   \frac{\www{p}_j\www{u}(r)}{\Gamma(1-\alpha_j)}dr\\
&+ \int^t_0\left(\int^{t-r}_0 A^{-1}S''(t-r-\xi)
(\xi^{-\alpha_j} - (t-r)^{-\alpha_j})d\xi\right)
\frac{\www{p}_j\www{u}(r)}{\Gamma(1-\alpha_j)}dr.
\end{align*}
By Theorem 2.2 in \cite{SY}, we obtain
\begin{align*}
\www{u}(t) 
=&-\int^t_0 A^{-1}S'(t-s)(B\cdot \nabla \www{u}(s)
  + b\www{u}(s) + F(s)) ds\\
+& \sum_{j=2}^{\ell} \int^t_0 A^{-1}S'(t-s)(t-s)^{-\alpha_j} 
   \frac{\www{p}_j\www{u}(s)}{\Gamma(1-\alpha_j)}ds\\
&+ \sum_{j=2}^{\ell}\int^t_0\left(\int^{t-r}_0 A^{-1}S''(t-r-\xi)
(\xi^{-\alpha_j} - (t-r)^{-\alpha_j})d\xi\right)
\frac{\www{p}_j\www{u}(r)}{\Gamma(1-\alpha_j)}dr.
\end{align*}
In the first and the second integrals on the right-hand side we make a 
change of variables $\tau = \frac{t-s}{t}$ and in the third integral
$(\xi,r) \mapsto (\tau,\eta)$ by 
$r = t - t\tau$, $\xi = t\tau\eta$, and we obtain 
\begin{align}\label{sol-u}
\widetilde{u}(t)
&=-t\int_0^1 A^{-1}S'(\tau t)\big(B\cdot\nabla \widetilde{u}((1-\tau)t)
+b\widetilde{u}((1-\tau)t)+F((1-\tau)t)\big)d\tau \nonumber\\
&+ \sum_{j=2}^{\ell}\frac{t^{1-\alpha_j}}{\Gamma(1-\alpha_j)}
    \int_0^1 A^{-1}S'(\tau t)\tau^{-\alpha_i}\widetilde{p}_j\widetilde{u}
\big((1-\tau)t\big)d\tau
    \nonumber\\
&+ \sum_{j=2}^{\ell}\frac{t^{2-\alpha_j}}{\Gamma(1-\alpha_j)}
\int_0^1 \int_0^1 A^{-1}S''\big((1-\eta)\tau t\big)(\eta^{-\alpha_j}-1)
\widetilde{p}_j\widetilde{u}\big((1-\tau)t\big)\tau^{1-\alpha_j}d\eta d\tau.
\end{align}

Furthermore, extending the variable $t$ in (\ref{sol-u}) from $(0,T)$ 
to the sector $\Omega_{\theta,T}$ and setting $\widetilde{u}_0=0$,
we define $\www u_{n+1}(z), n=0,1,\cdots$, $z\in \Omega_{\theta,T}$ 
as follows:
\begin{align} \label{def-induct-u_n}
\widetilde{u}_{n+1}(z)
=&-z\int_0^1 A^{-1}S'(\tau z)\big(B\cdot\nabla \widetilde{u}_n((1-\tau)z)
+b\widetilde{u}_n((1-\tau)z)+F(\cdot,(1-\tau)z)\big)d\tau\nonumber\\
&+ \sum_{j=2}^{\ell}\frac{z^{1-\alpha_j}}{\Gamma(1-\alpha_j)}
   \int_0^1 A^{-1}S'(\tau z)\tau^{-\alpha_j}\widetilde{p}_j\widetilde{u}_n
\big((1-\tau)z\big)d\tau
\nonumber\\
&+ \sum_{j=2}^{\ell}\frac{z^{2-\alpha_j}}{\Gamma(1-\alpha_j)}
\int_0^1 \int_0^1 A^{-1}S''\big((1-\eta)\tau z\big)(\eta^{-\alpha_j}-1)
\widetilde{p_j}\widetilde{u}_n\big((1-\tau)z\big)
      \tau^{1-\alpha_j}d\eta d\tau.
\end{align}

By (\ref{esti-S}) - (\ref{esti-S''}) we can inductively prove that 
$\www{u}_n(z)$ is analytic in $\Omega_{\theta}$ for any $n\in\mathbb{N}$.

Next we claim that  the following estimates hold:
\begin{align} \label{esti-induct-Au_n}
\|A\www u_{n+1}(z)-A \www u_n(z)\|_{L^2(\Omega)}
\leq M_1\frac{(CT^{\alpha_0}\Gamma(\alpha_0))^n}
{\Gamma(n\alpha_0+1)}, \quad z\in \Omega_{\theta,T}, \thinspace n=0,1,2,...
\end{align}
where 
$$
M_1 = T\Vert \lambda\Vert_{W^{2,\infty}(\Omega_{\theta,T})}, \quad
\alpha_0 = \min_{j=2,3,..., \ell}
\left\{\frac{\alpha_1}{2}, \alpha_1-\alpha_j\right\}.
$$

We now prove (\ref{esti-induct-Au_n}) by induction on $n$. 
Firstly, for $n=0$, integrating by parts and using (9) and (\ref{esti-S}), 
we see 
\begin{align}\label{esti-Au1-Au0}
&\|A\www u_{1}(z) - A\www u_0(z)\|_{L^2(\Omega)}
= \|A\www u_{1}(z)\|_{L^2(\Omega)}
= \left\Vert z\int^1_0 S'(\tau z)F((1-\tau)z) d\tau\right\Vert
_{L^2(\Omega)}\nonumber\\
=&\left\| S(\tau z) F(\cdot,(1-\tau)z)|_{\tau=0}^{\tau=1}
-\int_0^1S(\tau z)F'(\cdot,(1-\tau)z)(-z) d\tau \right\|_{L^2(\Omega)}\nonumber\\
\leq& \|S(z)F(\cdot,0)-F(\cdot,z)\|_{L^2(\Omega)} 
+ C\int_0^1 \|F'(\cdot,(1-\tau)z)\|_{L^2(\Omega)}d\tau\nonumber\\
\leq& CT\Vert \lambda\Vert_{W^{2,\infty}(\Omega_{\theta,T})}
= CM_1.
\end{align}
Next, for any $n\in\mathbb{N}$, taking the operator $A$ on both side 
of (\ref{def-induct-u_n}), and using (\ref{esti-S'}) and (\ref{esti-S''}) 
for the $z\in \Omega_{\theta,T}$, we can prove that
\begin{align*}
&\|A\www u_{n+1}(z)-A\www u_n(z)\|_{L^2(\Omega)}\\
\leq& C|z|\int_0^1 |\tau z|^{\frac{\alpha_1}{2}-1}
\|A\widetilde{u}_n((1-\tau)z)-A\widetilde{u}_{n-1}((1-\tau)z)\|_{L^2(\Omega)}
d\tau\\
&+C\sum_{j=2}^{\ell}|z|^{1-\alpha_j}\int_0^1|\tau z|^{\alpha_1-1}
\tau^{-\alpha_j}
\|A\widetilde{u}_n((1-\tau)z)-A\widetilde{u}_{n-1}((1-\tau)z)\|
_{L^2(\Omega)}d\tau\\
&+ C\sum_{j=2}^{\ell}|z|^{2-\alpha_j}\int_0^1\left (\int_0^1 
((1-\eta)\tau |z|)^{\alpha_1-2}(\eta^{-\alpha_j}-1)d\eta\right )
\tau^{1-\alpha_j}\|A\widetilde{u}_n((1-\tau)z)-A\widetilde{u}_{n-1}((1-\tau)z)
\|_{L^2(\Omega)}d\tau.
\end{align*}
Here by $B \in W^{1,\infty}(\Omega)$ and
$\Vert A^{\frac{1}{2}}v\Vert_{L^2(\Omega)}\le C\Vert v\Vert_{H^1(\Omega)}$
and $\Vert v\Vert_{H^2(\Omega)}\le C\Vert Av\Vert_{L^2(\Omega)}$ for 
$v \in D(A)$, we used
\begin{eqnarray*}
&& \Vert S'(\tau z)B\cdot (\nabla \www u_n - \nabla \www u_{n-1})
((1-\tau)z)\Vert_{L^2(\Omega)}
= \Vert A^{-\frac{1}{2}}S'(\tau z)A^{\frac{1}{2}}
(B\cdot (\nabla \www u_n - \nabla \www u_{n-1})((1-\tau)z))
\Vert_{L^2(\Omega)}\\
\le&& C\Vert A^{-\frac{1}{2}}S'(\tau z)\Vert_{L^2(\Omega)\to L^2(\Omega)}
\Vert B\cdot (\nabla \www u_n - \nabla \www u_{n-1})
((1-\tau)z)\Vert_{H^1(\Omega)}\\
\le&& C\Vert A^{-\frac{1}{2}}S'(\tau z)\Vert_{L^2(\Omega)\to L^2(\Omega)}
\Vert (A\www u_n - A\www u_{n-1})((1-\tau)z)\Vert_{L^2(\Omega)}.
\end{eqnarray*}

Noting that
$$
(1-\eta)^{\alpha-2}\leq \left(\frac{1}{2}\right)^{\alpha-2} \quad
\mbox{if $\eta\in \left[0,\frac{1}{2}\right]$}
$$
and
$$
\eta^{-\alpha}-1\leq C(1-\eta) \quad
\mbox{if $\eta\in \left[\frac{1}{2},1\right]$}
$$
for $0 < \alpha < 1$, we obtain
$$
\int_0^1 (1-\eta)^{\alpha_1-2}(\eta^{-\alpha_j}-1)d\eta
\le \int_0^\frac {1}{2} (1-\eta)^{\alpha_1-2}(\eta^{-\alpha_j}-1)d\eta
+ \int_\frac{1}{2}^1 (1-\eta)^{\alpha_1-1}\frac{\eta^{-\alpha_j}-1}{1-\eta}
d\eta
$$
$$
\le C\int_0^{\frac{1}{2}} (\eta^{-\alpha_j}-1) d\eta
+ C\int_{\frac{1}{2}}^1 (1-\eta)^{\alpha_1-1} d\eta<\infty.
$$
Therefore 
\begin{eqnarray*}
&& \Vert A\www u_{n+1}(z)-A\www u_n(z)\Vert_{L^2(\Omega)}\\
\leq &&C\left(|z|^{\frac{\alpha_1}{2}} + \sum_{j=2}^{\ell}\vert z\vert
^{\alpha_1-\alpha_j}\right)
\int^1_0 (\tau^{\frac{\alpha_1}{2}-1} + \sum_{j=2}^{\ell}\tau
^{\alpha_1-\alpha_j-1}) 
\Vert A\www u_{n}((1-\tau)z)-A\www u_{n-1}((1-\tau)z)
\Vert_{L^2(\Omega)} d\tau\\
\le && C\vert z\vert^{\alpha_0}\int^1_0 \tau^{\alpha_0-1}
\Vert A\www u_{n}((1-\tau)z)-A\www u_{n-1}((1-\tau)z)
\Vert_{L^2(\Omega)} d\tau, \quad z\in \Omega_{\theta,T}.
\end{eqnarray*}
Therefore 
\begin{equation}\label{esti-induct-int-Au_n}
\Vert A\www u_{n+1}(z)-A\www u_n(z)\Vert_{L^2(\Omega)}
\le C\vert z\vert^{\alpha_0}\int_0^1 \tau^{\alpha_0-1}
\|A\widetilde{u}_n((1-\tau)z)-A\widetilde{u}_{n-1}((1-\tau)z)\|
_{L^2(\Omega)}d\tau
\end{equation}
for $z \in \Omega_{\theta,T}$.
We note that $\int^1_0 t^{\alpha-1}(1-t)^{\beta-1} dt 
= \frac{\Gamma(\alpha)\Gamma(\beta)}{\Gamma(\alpha+\beta)}$ and
$\Gamma(\alpha+1) = \alpha\Gamma(\alpha)$ for $\alpha, \beta > 0$.
Iterating (\ref{esti-induct-int-Au_n}), in terms of (\ref{esti-Au1-Au0}), 
we obtain
$$
\Vert A\www u_2(z)-A\www u_1(z)\Vert_{L^2(\Omega)}
\le C\vert z\vert^{\alpha_0}\int^1_0 \tau^{\alpha_0-1}M_1 d\tau
= \frac{CM_1}{\alpha_0}\vert z\vert^{\alpha_0},
$$
\begin{eqnarray*}
&&\Vert A\www u_3(z)-A\www u_2(z)\Vert_{L^2(\Omega)}
\le C\vert z\vert^{\alpha_0}\int^1_0 \tau^{\alpha_0-1}
\frac{CM_1}{\alpha_0}\vert (1-\tau)z\vert^{\alpha_0} d\tau\\
= && \frac{(C\vert z\vert^{\alpha_0})^2M_1}{\alpha_0}
\frac{\Gamma(\alpha_0)\Gamma(\alpha_0+1)}{\Gamma(2\alpha_0+1)}
= \frac{(C\vert z\vert^{\alpha_0}\Gamma(\alpha_0))^2M_1}
{\Gamma(2\alpha_0+1)},
\end{eqnarray*}
and
$$
\Vert A\www u_4(z)-A\www u_3(z)\Vert_{L^2(\Omega)}
\le C\vert z\vert^{\alpha_0}\int^1_0 \tau^{\alpha_0-1}
\frac{M_1(C\vert (1-\tau)z\vert^{\alpha_0}\Gamma(\alpha_0))^2}
{\Gamma(2\alpha_0+1)} d\tau
= \frac{(C\vert z\vert^{\alpha_0}\Gamma(\alpha_0))^3M_1}
{\Gamma(3\alpha_0+1)}, \quad \mbox{etc.}
$$
Therefore similarly we obtain
$$
\Vert A\www u_{n+1}(z)-A\www u_n(z)\Vert_{L^2(\Omega)}
\le \frac{(C\vert z\vert^{\alpha_0}\Gamma(\alpha_0))^n}
{\Gamma(n\alpha_0+1)}M_1 
\le \frac{(CT^{\alpha_0}\Gamma(\alpha_0))^n}
{\Gamma(n\alpha_0+1)}M_1, \quad n=0,1,2,..., \thinspace
\forall z \in \Omega_{\theta,T}.
$$
Using (\ref{esti-Gamma}), we see that 
$$
\sum_{n=0}^{\infty} \frac{(CT^{\alpha_0}\Gamma(\alpha_0))^n}
{\Gamma(n\alpha_0+1)} < \infty.
$$
Hence the majorant test implies
$\sum_{n=0}^{\infty} \Vert A\www u_{n+1}(z) - A\www u_n(z)\Vert_{L^2(\Omega)}$
is convergent uniformly in $z \in \Omega_{\theta,T}$.
Therefore there exists $A{u}_*(z)\in L^2(\Omega)$ such that
$\|A\www u_n(z) - A{u}_*(z)\|_{L^2(\Omega)}$ tends to $0$ as
$n \rightarrow \infty$ uniformly in $z\in \Omega_{\theta,T}$.
Therefore $A{u}_*(z)$ is analytic in $\Omega_{\theta,T}$.
Moreover, since $T$ is arbitrarily chosen,  we deduce $A{u}_*(z)$
is analytic in the sector $\Omega_{\theta}$.

Next we prove (\ref{esti-u_C}).  In view of $p\le 0$ on $\overline{\Omega}$, 
we can prove
\begin{equation}\label{esti-max-u}
u(x,t) \le \max\{ 0, \max_{x\in\ppp\Omega, 0\le t \le T} g(x)\lambda(t)\}
\quad \mbox{for $x\in\overline{\Omega}$, $0\le t \le T$}.
\end{equation}
In fact, we can repeat the proof of Theorem 2 in Luchko \cite{Luchko2} which 
assumes that $p_1, ..., p_{\ell}$ are all constants and $p_1>0$, 
$p_j \ge 0$ for $j=2,..., \ell$.  Therefore (\ref{esti-max-u}) holds if 
$u$ is sufficiently smooth.  For our solution with the boundary value
$g(x)\lambda(t)$, applying an approximating argument similar to 
Theorems 4 and 5 in \cite{Podlubny}, we see (\ref{esti-max-u}) for the solutions
constructed in the theorem.

Replacing $u$ by $-u$ and applying (\ref{esti-max-u}), we obtain
$$
-u(x,t) \le \max\{0, \max_{x\in\ppp\Omega, 0\le t \le T} (-g(x)\lambda(t))\},
$$
that is,
$$
u(x,t) \ge \min\{ 0, \min_{x\in\ppp\Omega, 0\le t \le T} g(x)\lambda(t)\}
$$ 
for $x\in\overline{\Omega}$ and $0\le t \le T$.
With (\ref{esti-max-u}), we obtain
$$
\vert u(x,t)\vert \le \max_{x\in\ppp\Omega, 0\le t \le T} 
\vert g(x)\lambda(t)\vert
$$
for $x\in\overline{\Omega}$ and $0\le t \le T$.
Therefore the proof of (\ref{esti-u_C}) is completed.

Finally we show that ${u}_*(z)$ is the mild solution $\widetilde{u}$ 
to (\ref{equ-u'}) when the variable $z$ is restricted to $(0,T)$. 
In fact, denoting the imaginary part of $u_*(t)$,
$\forall t\in (0,T)$ as $\mbox{Im}\, u_*(t)$, we see that $v= \mbox{Re}\thinspace 
u_*(t)$ is a mild solution to the following
initial-boundary problem:
\begin{equation*}
\left\{ {\begin{array}{*{20}c}
{\partial_t^{\alpha_1}v+\sum_{j=2}^{\ell}p_j(x) \partial_t^{\alpha_j} v
   ={\rm div}(\frac{1}{p_1(x)}\nabla v)+ B(x)\cdot \nabla v + b(x)v},
\ (x,t)\in\Omega\times(0,T), \hfill \\
  {v(x,0)=0,\ x\in \Omega,} \hfill \\
  {v(x,t)=0,\ x\in \partial\Omega,\ t \in (0,T).} \hfill
 \end{array} } \right.
 \end{equation*}
Using the uniqueness result of the above problem (e.g., Theorem 2.4 in 
\cite{LY}), we have $\mbox{Im} \thinspace u_*(t)=0$, 
$\forall t\in(0,T)$.
Thus again by the uniqueness argument we see that $u_*(t)=u(t)$, 
$\forall t\in(0,T)$.
Consequently, we see that $\www{u}(t)=u(t)-\lambda(t)\widetilde{g}$ 
is analytic from $[0,T]$ to $H^2(\Omega)$ in view of the analyticity of 
$\lambda(t)$.  This completes the proof of the theorem.
\end{proof}

\section{Uniqueness for inverse boundary value problem}
The proof of Corollary \ref{buharik} is exactly the same as the proof 
of Theorem \ref{thm-uniqu}, and the only difference is that instead of 
the uniqueness result of \cite{SU}, we have to use the 
uniqueness result in \cite{IY}.  Thus it is sufficient to prove 
Theorem \ref{thm-uniqu}

\begin{proof}[Proof of Theorem \ref{thm-uniqu}]
We reduce the inverse problem to the corresponding inverse
boundary value problem for the Schr\"{o}dinger equation
\begin{equation*}
\left\{
\begin{array}{rl}
&\Delta v(x,s) + P_s(x)v(x,s) = 0, \quad x\in \Omega, \\
& v(x,s) = g(x), \quad x\in \partial\Omega,
         \end{array}
\right.
\end{equation*}
for all large $s>0$.  Here and henceforth we set
$$
P_s(x):=p(x)-\sum_{j=1}^{\ell} p_j(x) s^{\alpha_j}, \quad
Q_s(x):=q(x)-\sum_{j=1}^m q_j(x) s^{\beta_j}.
$$

Let $u_1(g)(x,t)$ and $u_2(g)(x,t)$ be the solutions to (\ref{equ-u})
with $(\ell,\vec{\alpha}, p_j,p)$ and $(m,\vec{\beta},q_j,q)$ respectively.
Since $\lambda(t)$ is $t$-analytic in $t>0$, Theorem \ref{thm-analy} implies
that $\Delta u_1(g)(x,t)$ and $\Delta u_2(g)(x,t)$ are $t$-analytic in $t>0$ 
for any fixed $x \in \overline\Omega$.  Therefore, since
$w\mapsto \frac{\ppp w}{\ppp\nu}$: $H^{\frac{3}{2}}(\ppp\Omega)
\rightarrow H^{\frac{1}{2}}(\ppp\Omega)$ is continuous,  
equality (\ref{DN}) implies
$$
\frac{\partial u_1(g)}{\partial\nu}(x,t) 
= \frac{\partial u_2(g)}{\partial\nu}(x,t),
\quad x \in \partial\Omega, \thinspace 0 < t < \infty \quad
\mbox{for $g \in C^{\infty}(\ppp\Omega)$}.
$$
Let $(Lu)(x,s) := \int^{\infty}_0 e^{-st}u(x,t)dt$ be 
the Laplace transform of $u(x,t)$ in $t$ for each fixed 
$x \in \overline{\Omega}$.
By (\ref{esti-u_C}) in Theorem \ref{thm-analy} and assumption 
$\vert \lambda(t) \vert \le C_0e^{C_0t}$ for $t>0$, 
we see that $\vert u(x,t)\vert \le Ce^{C_0t}$ for $t>0$, where
$C>0$ is a constant and is independent of $t>0$ and $x\in \Omega$.
Therefore $(Lu_k(g))(x,s)$, $k=1,2$, exist
for $s > C_1$ where $C_1>0$ is some constant depending only on 
$\lambda$.  Using $u_k(g)(x,0) = 0$, by \cite{Podlubny}, we have
$$
L(\partial_t^{\alpha}u_k(g))(x,s) = s^{\alpha}(Lu_k(g))(x,s),
\quad s > C_1, \thinspace k=1,2.
$$
Therefore by the fractional diffusion equations themselves, it follows that
$L(\Delta u_k(g))(x,s)$, $k=1,2$, exist for $s > C_1$.
Hence
$$
\left\{
\begin{array}{rl}
&\Delta L(u_1(g))(x,s) + P_s(x)L(u_1(g))(x,s) = 0, \quad
x\in \Omega, \thinspace s > C_1,\\
& L(u_1(g))(x,s) = (L\lambda)(s)g(x), \quad x\in \partial\Omega,
\thinspace s > C_1,
\end{array}
\right.
$$
$$
\left\{
\begin{array}{rl}
&\Delta L(u_2(g))(x,s) + Q_s(x)L(u_2(g))(x,s) = 0, \quad
x\in \Omega, \thinspace s > C_1,\\
& L(u_2(g))(x,s) = (L\lambda)(s)g(x), \quad \forall x\in \partial\Omega,\quad\mbox{and}\,\,
\thinspace \forall s > C_1,
\end{array}
\right.
$$
and
$$
\frac{\partial L(u_1(g))}{\partial\nu}(x,s)
= \frac{\partial L(u_2(g))}{\partial\nu}(x,s),
\quad \forall x\in \partial\Omega,\quad\mbox{and}\,\,
\thinspace \forall s > C_1.
$$

Next we consider the following two boundary value problems
\begin{equation}\label{equ_v1}
\left\{
\begin{array}{rl}
&\Delta v_1(x,s) + P_s(x)v_1(x,s) = 0, \quad
x\in \Omega, \thinspace s > C_1,\\
& v_1(x,s) = g(x), \quad \forall x\in \partial\Omega,\quad\mbox{and}\,\,
\thinspace \forall s > C_1.
\end{array}
\right.
\end{equation}
and
\begin{equation}\label{equ_v2}
\left\{
\begin{array}{rl}
&\Delta v_2(x,s) + Q_s(x)v_2(x,s) = 0, \quad
x\in \Omega, \thinspace s > C_1,\\
& v_2(x,s) = g(x), \quad \forall x\in \partial\Omega,\quad\mbox{and}\,\,
\thinspace \forall s > C_1.
\end{array}
\right.
\end{equation}
Then we define their Dirichlet-to-Neumann maps $\Lambda(P_s)$ and
$\Lambda(Q_s)$ by 
$$
\Lambda(P_s)g:=\frac{\partial v_1(g)}{\partial\nu}|_{\partial\Omega},\quad
\Lambda(Q_s)g:=\frac{\partial v_2(g)}{\partial\nu}|_{\partial\Omega}.
$$
Now we prove that there exists a subset 
$\sigma \subset (C_1,\infty)$ such that $\sigma$ contains a non-empty 
open interval and 
\begin{equation}\label{DN-DN'}
\Lambda(\ell,\vec{\alpha},p_j,p)g=\Lambda(m,\vec{\beta},q_j,q)g\Longrightarrow
\Lambda(P_s)g = \Lambda(Q_s)g \quad \mbox{for all 
$g \in C^{\infty}(\partial\Omega)$ and $s\in \sigma$}.
\end{equation}
In fact, $(L\lambda)(z)$ is analytic in Re $z > C_1$ and 
$\{s; \thinspace (L\lambda)(s) = 0, \thinspace s>C_1\}$ has no
accumulation points except for $\infty$.  Therefore 
$\sigma:= (C_1,\infty) \setminus \{s; \thinspace (L\lambda)(s) = 0, 
\thinspace s>C_1\}$ contains a non-empty open interval.  Then 
we can set $\www v_j(g)(x,s) = \frac{L(u_j(g))(x,s)}{(L\lambda)(s)}$ for
$j=1,2$ and $s \in \sigma$.  It is not very difficult to see that 
$\widetilde{v}_1(g)$ and $\www v_2(g)$ are the solutions to (\ref{equ_v1}) 
and (\ref{equ_v2}) respectively. From the uniqueness of the boundary value 
problem, we see that $\widetilde{v}_j(g)=v_j(g)$, $j=1,2$ for $s \in \sigma$.

Here by the density of $C^{\infty}(\ppp\Omega)$ in 
$H^{\frac{1}{2}}(\ppp\Omega)$ and the continuity of 
$\Lambda(P_s) : H^{\frac{1}{2}}(\ppp\Omega) \longrightarrow
H^{-\frac{1}{2}}(\ppp\Omega)$, we see that (\ref{DN-DN'}) holds for all 
$g \in H^{\frac{1}{2}}(\ppp\Omega)$.

Therefore the uniqueness \cite{SU} by Dirichlet-to-Neumann map in
determining a potential, we see that $P_s(x) = Q_s(x)$ for all $x\in\Omega$ and
$s \in \sigma$.  Since $\sigma$ contains a non-empty open 
interval, we obtain $\ell = m$, $\vec{\alpha}=\vec{\beta}$, $p_j=q_j$, $1\le j \le \ell$ 
and $p=q$.  Thus the proof of the theorem is completed.

\end{proof}

{\bf Acknowledgement.}
The first author thanks the Chinese Scholarship Council (CSC) 
and Leading Graduate Course for Frontiers of Mathematical Sciences and 
Physics (FMSP, The University of Tokyo) for financial supports.

\end{document}